\documentclass[final,nospthms]{svjour3}
\smartqed
\usepackage{graphicx}
\usepackage{color}
\usepackage{hyperref}
\usepackage{algorithmic}
\usepackage{xcolor}
\usepackage{enumitem,multicol,amsmath,amsopn,mathtools} 
\usepackage{mathptmx} 
\usepackage{amsfonts} 
\usepackage{mathrsfs,amsmath}
\usepackage{amsthm,bm}
\DeclareMathAlphabet{\mathcalligra}{T1}{calligra}{m}{n}
\newtheorem{theorem}{Theorem}[section]

\newtheorem{remark}[theorem]{Remark}
\newtheorem{lemma}[theorem]{Lemma}
\newtheorem{corollary}[theorem]{Corollary}
\newcommand{\ontop}[2]{\genfrac{}{}{0pt}{}{#1}{#2}}
\newcommand{\NN}{\mathbb{N}}
\newcommand{\ZZ}{\mathbb{Z}}
\newcommand{\RR}{\mathbb{R}}
\newcommand{\PP}{\mathbb{P}}
\newcommand{\EE}{\mathbb{E}}
\newcommand{\var}{\text{var}}

\newcommand{\mix}{t_{\mathrm{mix}}} 
\newcommand{\rel}{t_{\mathrm{rel}}}

\newcommand{\sG}{\mathcal{G}}
\newcommand{\sS}{\mathcal{S}}
\newcommand{\sE}{\mathcal{E}}
\newcommand{\sL}{\mathcal{L}}
\newcommand{\sM}{\mathcal{M}}
\newcommand{\sP}{\mathcal{P}}

\newcommand{\sC}{\mathcal{C}}
\newcommand{\cycle}{\mathcal{C}_0}
\newcommand{\noedges}{\mathbf{0}}

\newcommand{\indicator}{\mathbf{1}}

\newcommand{\defect}{\mathcal{C}_2}
\newcommand{\doubledefect}{\mathcal{C}_4}
\newcommand{\worm}{\mathcal{W}}
\newcommand{\cp}{\sP}
\newcommand{\symdif}{\triangle}
\newcommand{\maxdeg}{\Delta}

\newcommand{\e}{\mathrm{e}}

\numberwithin{equation}{section}
\begin{document}
\title{The worm process for the Ising model is rapidly mixing}
\titlerunning{The worm process for the Ising model is rapidly mixing}        
\author{Andrea Collevecchio \and Timothy M. Garoni \and Timothy Hyndman \and Daniel Tokarev}
\authorrunning{Collevecchio, Garoni, Hyndman, \& Tokarev}
\institute{Andrea Collevecchio \and Timothy M. Garoni \and Timothy Hyndman \and Daniel Tokarev \at
  School of Mathematical Sciences, Monash University, Clayton, VIC, 3800, Australia\\
  Tel.: +61 3 9905 4465, Fax: +61 3 9905 4403\\
  \email{andrea.collevecchio@monash.edu}\\
  \email{tim.garoni@monash.edu}\\
  \email{tim.hyndman@monash.edu}\\
  \email{daniel.tokarev@monash.edu}
}
\date{Received: September 10, 2015 / Accepted: }
\maketitle
\begin{abstract}
We prove rapid mixing of the worm process for the zero-field ferromagnetic Ising model, on all finite connected graphs, and at all
temperatures. As a corollary, we obtain a fully-polynomial randomized approximation scheme for the Ising susceptibility, and for 
a certain restriction of the two-point correlation function.
\end{abstract}

\section{Introduction}
\label{sec:introduction}
The ferromagnetic Ising model on finite graph $G=(V,E)$ at inverse temperature $\beta\ge0$ with external field $h\in\RR$ is defined by the Gibbs measure
\begin{equation}
  \PP_{G,\beta,h}(\sigma) = \frac{1}{Z_{G,\beta,h}} \e^{-\beta H_{G,h}(\sigma)}, \qquad\sigma\in\{-1,+1\}^V,
  \label{Ising measure}
\end{equation}
with Hamiltonian 
\begin{equation}
  H_{G,h}(\sigma) = -\sum_{ij\in E}\sigma_i\sigma_j - h \sum_{i\in V}\sigma_i,
  \label{Ising Hamiltonian}
\end{equation}
and partition function
\begin{equation}
  Z_{G,\beta,h} = \sum_{\sigma\in\{-1,1\}^V} \e^{-\beta H_{G,h}(\sigma)}.
\end{equation}

The central computational challenge associated with the measure~\eqref{Ising measure} is to estimate certain moments of physically relevant
random variables. Chief amongst such moments are the two-point correlation function and susceptibility. While the problem is simple to pose,
studying such quantities on large graphs is typically a computationally demanding task, and it is therefore not
surprising that the Markov-chain Monte Carlo method is one of the most common approaches employed. This
computational intractability can in fact be made precise in the language of complexity theory. It was recently
established~\cite{SinclairSrivastava14} that computing the susceptibility of the Ising model is \#P-hard, and the \#P-hardness of the
two-point correlation function then follows as an immediate corollary. Various \#P-hardness results are also known for the Ising partition
function~\cite{JerrumSinclair93,DyerGreenhill00,BulatovGrohe05}.

In order for a particular Markov chain to provide efficient estimators, it is necessary that it converges rapidly to
stationarity. Consequently, in addition to the single-spin Glauber process~\cite{Martinelli99}, which has a direct physical interpretation,
a host of more sophisticated processes have been devised, with the aim of improving the efficiency of the resulting estimators. These
processes typically augment, or replace, the spin measure~\eqref{Ising measure} with a particular graphical representation. For example,
Sweeny~\cite{Sweeny83} proposed studying the zero-field Potts model using the single-bond Glauber process for the Fortuin-Kasteleyn
model~\cite{FortuinKasteleyn72,Grimmett06}, while the Swendsen-Wang (SW) process~\cite{SwendsenWang87} simulates a natural
coupling~\cite{EdwardsSokal88} of the Fortuin-Kasteleyn model and zero-field Potts model. Similarly, Jerrum and
Sinclair~\cite{JerrumSinclair93} studied a single-bond Metropolis process for the high-temperature representation of the Ising model in a
strictly positive field.  Prokof'ev and Svistunov~\cite{ProkofevSvistunov01} also considered a space of Ising high temperature graphs,
however their \emph{worm} process applies to the case of zero field, and uses a novel choice of local moves. Given this abundance of
Markov-chain Monte Carlo algorithms for the Ising model, one naturally seeks to understand and compare their efficiencies.

A key quantity for characterizing the rate of convergence of reversible Markov chains is the relaxation time. For a reversible finite Markov
chain with transition matrix $P$, let $\lambda_{\star}$ denote the maximum absolute value of the non-trivial eigenvalues of $P$. The
relaxation time of $P$ is then simply the reciprocal of the (absolute) spectral gap
\begin{equation}
  \rel := \frac{1}{1-\lambda_{\star}}.
  \label{relaxation time definition}
\end{equation}
Each of the abovementioned types of process (single-spin Glauber, Swendsen-Wang etc), can be applied to the Ising model on any finite graph.
Let $\mathcal{F}$ be a given infinite family of finite graphs. For a given type of process, and a given choice of Ising parameters
$\beta,h$, for each $G\in\mathcal{F}$, we can consider the corresponding process for the Ising model on $G$ with parameters $\beta,h$. For
each $G\in\mathcal{F}$, we then have a corresponding value of $\rel$. If the map $G\mapsto \rel$ can be bounded above by a polynomial in
$|V(G)|$, one says that type of process is \emph{rapidly mixing} on $\mathcal{F}$, for the given choice of parameters $\beta,h$. Otherwise,
the mixing is said to be \emph{torpid}.

There is a vast literature discussing the mixing of the Ising Glauber process; see for example the survey~\cite{Martinelli99}. It has
recently been shown~\cite{MosselSly13} that if $(\maxdeg-1)\tanh\beta<1$, then the Ising Glauber process is rapidly mixing for all graphs of
maximum degree $\maxdeg$. This result is tight in the sense that if $(\maxdeg - 1) \tanh \beta > 1$, then, with high probability, the mixing
of the zero-field Ising Glauber process on random $\maxdeg$-regular graphs is torpid~\cite{GerschenfeldMontanari07}. For finite boxes in
$\ZZ^2$, the zero-field Glauber process is rapidly mixing above the critical temperature~\cite{MartinelliOlivieri94}, and \emph{at} the
critical temperature~\cite{LubetzkySly12}, but torpid below the critical temperature~\cite{CesiGuadagniMartinelliSchonmann96}. The same
behaviour is known to occur in zero field on the complete graph, where the mixing is very well
understood~\cite{DingLubetzkyPeres09IsingGlauber} at all temperatures.

The Swendsen-Wang process has also been the subject of significant study. We focus on results for the Ising case. Mixing on the complete
graph is very well understood~\cite{CooperDyerFriezeRue00,LongNachmiasNingPeres14} at all temperatures.  For graphs with bounded
maximum degree, rapid mixing was established for all sufficiently high temperatures in~\cite{CooperFrieze99}. More recently, it has been
shown that for such graphs, rapid mixing of the single-site Glauber process implies rapid mixing of the Swendsen-Wang
process~\cite{Ullrich12,Ullrich12Thesis}. In particular, given the single-site Glauber results mentioned above, this established rapid
mixing of the Swendsen-Wang process on finite boxes in $\ZZ^2$ for all temperatures at and above the critical temperature.  Comparison
results between the Swendsen-Wang process and the single-bond Glauber process for the Fortuin-Kasteleyn model have also been recently
established~\cite{Ullrich14,Ullrich12Thesis}, which show that the single-bond process is rapidly mixing iff the Swendsen-Wang process is rapidly mixing. As
a consequence, by exploiting the duality of the Fortuin-Kasteleyn model, this established rapid mixing of the Swendsen-Wang process on boxes in
$\ZZ^2$ at all temperatures below the critical temperature.

Currently, perhaps the best understood of the above mentioned processes, however, is the Jerrum-Sinclair-Ising (JSI) process~\cite{JerrumSinclair93},
which is known to be rapidly mixing on all graphs, at all temperatures, provided the field is strictly positive. Specifically:
\begin{theorem}[Jerrum-Sinclair~\cite{JerrumSinclair93,Sinclair92}]
The relaxation time of the JSI process on any finite connected graph with $m$ vertices, at any temperature, and in a field $h>0$ satisfies
$$
\rel^{\mathrm{JS}} \le 2m^2\mu^{-4},
$$
where $\mu=\tanh(h)$.
\label{JS theorem}
\end{theorem}
\begin{remark}
The divergence of the upper bound for $\rel^{\mathrm{JS}}$ as $\mu\to0$ is to be expected, given that the process is not irreducible when $h=\mu=0$.
\end{remark}
As a consequence of Theorem~\ref{JS theorem}, in~\cite{JerrumSinclair93} a \emph{fully-polynomial randomized approximation scheme} (fpras) was
constructed for the Ising partition function, as well as the mean energy and magnetization.
An fpras~\cite{KarpLuby89,JerrumSinclair93,Jerrum03,Vazirani01} for an Ising observable $X_{G,\beta,h}$ is a randomized algorithm which,
given as input a problem instance $(G,\beta,h)$ and real numbers $\epsilon,\eta\in(0,1]$, outputs a random number $Y$ satisfying
\begin{equation}
  \PP[(1-\epsilon) X_{G,\beta,h} \le Y \le (1+\epsilon)X_{G,\beta,h}] \ge 1-\eta,
\end{equation}
in a time which is at most a polynomial in $|V|$, $\epsilon^{-1}$ and $\eta^{-1}$. Notably, despite the fact that the JSI process is not
irreducible when $h=0$, the fprases constructed in~\cite{JerrumSinclair93} are valid for all $h\ge0$, including $h=0$.
In~\cite{RandallWilson99}, a fully-polynomial approximate generator from the measure~\eqref{Ising measure} is described, which uses
successive calls of the partition function fpras presented in~\cite{JerrumSinclair93} to generate approximate samples from the
Fortuin-Kasteleyn measure, from which Ising samples can then be generated using the Edwards-Sokal coupling~\cite{EdwardsSokal88}.

In the present article, we study the mixing of the worm process. Like the JSI process, the worm process is based on the high temperature
expansion of the Ising model. However, the key difference between the two approaches is that while~\cite{JerrumSinclair93} considered only
strictly positive fields in their high temperature expansions, \cite{ProkofevSvistunov01} considered only the case of strictly zero
field. As a consequence, while the configuration space of the measure considered in~\cite{JerrumSinclair93} consists of the full edge space
of $G$, the configuration space introduced in~\cite{ProkofevSvistunov01} consists of spanning subgraphs subject to strong constraints on
the vertex degrees. While this may at first sight seem to be a disadvantage, the worm process provides a simple and natural method of
sampling from this non-trivial space of combinatorial objects, and gives rise to particularly natural estimators for the Ising
susceptibility, and two-point correlation function.

The worm process was first introduced in the context of quantum Monte Carlo in~\cite{ProkofevSvistunovTupitsyn98}, and classical versions
were subsequently described in~\cite{ProkofevSvistunov01}.  In~\cite{DengGaroniSokal07_worm}, a numerical study of the worm process
concluded that it is currently the most efficient method known for estimating the susceptibility of the three-dimensional Ising
model. Numerical evidence presented in~\cite{Wolff09a} also suggests it provides a very efficient method for studying the Ising two-point
correlation function. Applications and extensions of the worm process now constitute an active topic in computational physics; see for
example~\cite{HitchcockSorensen04,Wang05,BurovskiMachtaProkofevSvistunov06,BercheChatelainDhallKennaLowWalter08,WinterJankeSchakel08,
  Wolff09,Wolff09b,Wolff10a,Wolff10b,JankeNeuhausSchakel10,KorzecVierhausWolff11,DrakeMachtaDengAbrahamNewman12,KorzecWolff12,
  WalterBarkema15}. To our knowledge, however, no rigorous results have previously been reported on the mixing of the worm process.

In the current article, we prove that the worm process for the zero-field Ising model is rapidly mixing on all connected graphs, at all
temperatures. As a corollary, we show that the standard worm estimators for the susceptibility and two-point correlation functions used by
computational physicists, which are simply sample means of natural observables, define fully-polynomial randomized approximation schemes for
these quantities.  In the latter case, we restrict attention to the correlations between pairs of vertices whose graph distance is bounded
above by some fixed distance $k\in\NN$.

\subsection{Outline}
Let us outline the remainder of this article. Section~\ref{sec:preliminaries} establishes some preliminary notation and terminology that we
shall use throughout. Our main results on the relaxation and mixing times are then stated in Section~\ref{sec:main results}. In
Section~\ref{sec:worm measure and process} we introduce and study the Prokof'ev-Svistunov (PS) measure, and describe its relationship to the
Ising model. The PS measure is the stationary distribution of the worm process, whose definition we give in Section~\ref{sec:worm process}.
The proof of Theorem~\ref{thm:mixing time bound} is presented in Section \ref{sec:proof of main theorem}. Finally, Section~\ref{sec:fpras}
describes how to use the mixing time bound to construct fully-polynomial randomized approximation schemes for the Ising susceptibility and
two-point correlations.

\subsection{Preliminaries}
\label{sec:preliminaries}
For a finite graph $G=(V,E)$, we set $n=|V|$ and $m=|E|$, and denote the maximum degree by $\maxdeg$.  For any pair of vertices $u,v\in V$,
we let $d(u,v)$ denote the graph distance between them.  To avoid trivialities, we shall always assume $m\ge1$. For simplicity, as the underlying graph
$G$ can be considered fixed throughout, we suppress explicit mention of $G$ in our notation.

We shall be interested in certain random spanning subgraphs of $G$.  To avoid confusion, we shall denote the empty set in $2^E$ by
$\noedges$, as distinct from the empty set in the corresponding sigma algebra $2^{2^E}$, which we simply denote $\emptyset$. Since $2^E$
forms a vector space over $\ZZ_2$, in which $\noedges$ is the zero vector, this notation seems quite natural; see e.g.~\cite{Diestel05}.

If a vertex in a given graph has odd degree, then we shall call it an \emph{odd vertex}. For $A\subseteq E$, the set of odd
vertices in the spanning subgraph $(V,A)$ will be denoted by $\partial A$. For a given set $A \subseteq E$, we let $d_A(v)$ be the degree of
$v\in V$ in the spanning subgraph $(V,A)$, i.e.
\begin{equation}
d_A(v) := \#\{u\in V \colon uv\in A\}.
\end{equation}
For $n\in\NN$, we write $[n]:=\{1,2,\ldots,n\}$.

We emphasize that henceforth, in all that follows, we shall focus entirely on the case of zero field, in which we set $h=0$ in~\eqref{Ising
  Hamiltonian}. In addition, since~\eqref{Ising measure} is trivial when the temperature is either zero or infinite, and to avoid trivial
but tedious technicalities, we shall confine our attention at all times to the case $0<\beta <\infty$.

\subsection{Statement of Main Results}\label{sec:main results}
Given an aperiodic and irreducible Markov chain with finite state space $\Omega$, transition matrix $P$, and stationary distribution $\pi$,
and given a prescribed $\delta\in(0,1)$, one defines (see e.g.~\cite{LevinPeresWilmer09,Jerrum03,AldousFill14}) the \emph{mixing time} from state
$\omega\in\Omega$ to be
\begin{equation}
  \mix(\omega,\delta) := \min\{t\in \NN :\|P^t(\omega,\cdot)-\pi\|\le\delta\},
  \label{mixing time definition}
\end{equation}
where $\| \cdot\|$ denotes total variation distance. One further defines $\mix(\delta):=\max_{\omega\in\Omega}\mix(\omega,\delta)$. 

\begin{theorem}
\label{thm:mixing time bound}
  Consider the zero-field ferromagnetic Ising model on a finite connected graph at inverse temperature $\beta>0$,
  and let $x=\tanh\beta$. The corresponding worm process, defined by the transition matrix~\eqref{worm transition matrix}, satisfies:
  $$
  \rel \le 4 \maxdeg\,m\,n^4,
  $$
  and for any $\delta\in(0,1)$
\begin{align*}
\mix(\noedges,\delta) &\le 4\left[\log2 + \frac{\log\delta^{-1}}{m}\right] \maxdeg\,m^2\,n^4,\\
\mix(\delta) &\le 4\left[\log\left(\frac{2}{x}\right)+\frac{\log2\delta^{-1}}{m}\right] \maxdeg\,m^2\,n^4.
\end{align*}
\end{theorem}
As discussed in Section~\ref{sec:fpras}, the state $\noedges$ is the natural state in which to initialize the worm
process, which explains the special treatment afforded it in Theorem~\ref{thm:mixing time bound}.

\section{Definition of the PS measure and worm process}
\label{sec:worm measure and process}
\subsection{High temperature expansions}\label{sec:high temperature expansion}
We begin by recalling the standard high-temperature expansion for the Ising correlation functions (see e.g.~\cite{Thompson79}), and note some
simple consequences of it that will form key ingredients in our proof Theorem~\ref{thm:mixing time bound}. We begin with some notation. 

Given $x\in(0,1)$ and a finite graph $G=(V,E)$, we define the following measure on $2^{E}$,
\begin{equation}
\lambda_{x}(S):=\sum_{A\in S} x^{|A|}, \quad\qquad S\subseteq 2^E.
\end{equation}
We emphasize that while $\lambda_{x}(\emptyset)=0$, by contrast $\lambda_{x}(\noedges)=1$.
In addition, for any $W\subseteq V$ we let
\begin{equation}
\sC_{W} := \{A\subseteq E:\partial A= W\},
\end{equation}
and, in a slight abuse of notation, for any pair of vertices $u,v$ we will write $\sC_{uv} = \sC_{\{u,v\}}$,
and for $k\in[n]$ we will write
\begin{equation}
\sC_{k} := \bigcup_{\ontop{W\subseteq V}{|W|=k}}\sC_W.
\end{equation}

  \begin{lemma}
  \label{lem:high temperature expansion}
  Consider a finite graph $G=(V,E)$ and inverse temperature $\beta >0$. Let $\EE_{\beta}$ denote expectation with respect to the zero-field Ising
  measure on $G$, defined in \eqref{Ising measure}, and let $\lambda_x$ denote the corresponding measure on $2^E$ with
  $x=\tanh(\beta)$. Then, for any $W\subseteq V$, we have
  $$
  \EE_{\beta}\,\left(\prod_{i\in W}\sigma_i\right) = \displaystyle\frac{\lambda_{x}(\sC_W)}{\lambda_{x}(\sC_{0})}.
  $$
  \end{lemma}
  \begin{proof}
    Begin by observing that for $\sigma\in\{-1,1\}^V$ we have
       \begin{align*}
         \prod_{ij\in E}\e^{\beta\sigma_i\sigma_j} 
         &= \prod_{ij\in E}[\cosh(\beta) + \sigma_i\sigma_j\sinh(\beta)]\\
         &= \cosh^m(\beta)\,\prod_{ij\in E}[1 + \sigma_i\sigma_j x]\\
         &= \cosh^m(\beta)\,\sum_{A\subseteq E}x^{|A|} \prod_{ij\in A}\sigma_i\sigma_j\\
         &= \cosh^m(\beta)\,\sum_{A\subseteq E}x^{|A|} \prod_{i\in V}\sigma_i^{d_A(i)}.
       \end{align*}
       This implies
       \begin{align*}
         Z_\beta &= \sum_{\sigma\in\{-1,1\}^V}\,\prod_{ij\in E}e^{\beta\sigma_i\sigma_j}\\
         &=\cosh^m(\beta)\sum_{A\subseteq E}\,x^{|A|}\,\prod_{i\in V}\sum_{\sigma_i\in\{-1,+1\}}\sigma_i^{d_A(i)}\\
         &=2^n\,\cosh^m(\beta)\lambda_x(\cycle)
       \end{align*}
       since the sum $\sum_{\sigma_i\in\{-1,+1\}}\sigma_i^k$ equals $0$ for $k$ odd, and 2 for $k$ even. Likewise,
       \begin{align*}
         \EE_{\beta}\,\left(\prod_{i\in W}\sigma_i\right) 
         &= \frac{1}{Z_\beta}\sum_{\sigma\in\{-1,1\}^V}\, \prod_{j\in W}\sigma_j \,\prod_{ij\in E} \e^{\beta\sigma_i\sigma_j}\\
         &= 
         \frac{1}{2^n\lambda_x(\cycle)}
           \sum_{A\subseteq E}\,x^{|A|}\,
           \left(\prod_{i\in V\setminus W}\,\,\sum_{\sigma_i\in\{-1,+1\}}\,\sigma_i^{d_A(i)}\right)
           \left(\prod_{j\in W}\,\,\sum_{\sigma_j\in\{-1,+1\}}\sigma_j^{d_A(j)+1}\right)
         \\
         &= \frac{\lambda_{x}(\sC_W)}{\lambda_{x}(\cycle)}.
       \end{align*}
\end{proof}

  \begin{corollary} \label{cor:lambda of Ck over lambda of C0}
    Consider a finite graph $G=(V,E)$ and $x\in(0,1)$. For any $W\subseteq V$, and any $k\in[n]$
    $$
    \frac{\lambda_{x}(\sC_W)}{\lambda_{x}(\cycle)} \le 1,\qquad\qquad
    \frac{\lambda_{x}(\sC_k)}{\lambda_{x}(\cycle)} \le \binom{n}{k}.
    $$
    \begin{proof}
      Since $\left(\prod_{j\in W}\sigma_j\right)\le 1$, we have $\EE_{\beta}\left(\prod_{j\in W}\sigma_j\right)\le1$, and
      Lemma~\ref{lem:high temperature expansion} then immediately implies the first stated result. To obtain the second result, we note that
      since ${\sC_W\cap\sC_{W'}=\emptyset}$ whenever $W\neq W'$, the first result implies
      $$
      \frac{\lambda_{x}(\sC_k)}{\lambda_{x}(\cycle)} 
      = \sum_{W\in\binom{V}{k}} \frac{\lambda_{x}(\sC_W)}{\lambda_{x}(\cycle)} \le \binom{n}{k},
      $$
      where $\binom{V}{k}$ denotes the set of all subsets of $V$ of size $k$.
    \end{proof}
\end{corollary}
  
\subsection{Prokof'ev-Svistunov measure}\label{sec:PS measure}
One simple consequence of Lemma~\ref{lem:high temperature expansion} is that the variance of the Ising magnetization $\sM(\sigma) = \sum_{i\in V}
\sigma_i$ satisfies
\begin{equation}
\var_{\beta} \sM 
=\sum_{u,v\in V}\EE_{\beta}(\sigma_u\sigma_v)
=\frac{n\lambda_{x}(\sC_{0})+2\lambda_{x}(\sC_2)}{\lambda_{x}(\sC_{0})}.
\end{equation}
The susceptibility $\chi_{\beta} = \displaystyle\frac{1}{n}\var_{\beta}(\sM)$ therefore satisfies
\begin{equation}
\frac{1}{\chi_{\beta}}
= 
\frac{n\lambda_{x}(\sC_{0})}{n\lambda_{x}(\sC_{0})+2\lambda_{x}(\sC_2)}.
\end{equation}
 This motivates introducing the configuration space $\worm=\cycle\cup\defect$, and the probability measure $\pi_x$ defined by
\begin{equation}
\begin{split}
\pi_{x}(A) &:= 
\frac{x^{|A|}\,\psi(A)}{n\lambda_{x}(\sC_{0})+2\lambda_{x}(\defect)},\\
\psi(A) &:=
 \begin{cases}
   n, & A\in \cycle,\\
   2, & A\in \defect.
 \end{cases}
\end{split}
\label{worm measure}
\end{equation}
Consideration of the probability space $(\worm,\pi_{x})$ was first proposed in~\cite{ProkofevSvistunov01}. We refer to
$\pi_{x}$ as the Prokof'ev-Svistunov (PS) measure.
Several Ising observables can be expressed neatly in terms of the PS measure, including the susceptibility
\begin{equation}
\chi_{\beta} = \frac{1}{\pi_{x}(\cycle)},
\label{susceptibility identity}
\end{equation}
and the two-point correlation function
\begin{equation}
\EE_{\beta}(\sigma_u\sigma_v) = \frac{n}{2}\frac{\pi_{x}(\sC_{uv})}{\pi_{x}(\cycle)}.
\label{correlation identity}
\end{equation}

\subsection{Worm process}\label{sec:worm process}
The worm process is a Markov chain on the state space $\worm$, constructed by metropolizing the following proposals with respect to the PS
measure~\eqref{worm measure}.
\begin{itemize}
\item If $A\in\cycle$:
  \begin{itemize}
    \item choose a uniformly random vertex $v\in V$
    \item choose a uniformly random neighbour $u\sim v$
    \item propose $A\mapsto A\symdif uv$
  \end{itemize}
\item If $A\in\defect$:
  \begin{itemize}
    \item choose a uniformly random odd vertex $v\in \partial A$
    \item choose a uniformly random neighbour $u\sim v$
    \item propose $A\mapsto A\symdif uv$
  \end{itemize}
\end{itemize}
Here $A\symdif uv$ denotes symmetric difference of $A$ and the edge $uv$; i.e. if $uv\in A$ we propose to delete it, while if $uv\not\in A$
we propose to add it. We begin by observing that these proposals are indeed well defined, in the sense that we necessarily have $A\symdif
uv\in\worm$. To see this, we first note the following elementary lemma.
\begin{lemma}
\label{lem:boundary of symdif is symdif of boundaries}
Consider a finite graph $G=(V,E)$. If $A,B\subseteq E$, then $\partial(A\symdif B)=(\partial A)\symdif(\partial B)$.
\begin{proof}
If $v\in V$, then
$$
d_{A\symdif B}(v) = d_A(v) + d_B(v) -2 d_{A\cap B}(v).
$$
Therefore, $d_{A\symdif B}(v)$ is odd iff $d_A(v)+d_B(v)$ is odd. The latter holds iff either $d_A(v)$ is odd and $d_B(v)$ is even, or
$d_A(v)$ is even and $d_B(v)$ is odd. Consequently, $v\in\partial(A\symdif B)$ iff either $v\in\partial A$ and $v\not\in\partial B$, or
$v\in\partial B$ and $v\not\in\partial A$. Therefore, $v\in\partial(A\symdif B)$ iff $v\in(\partial A)\symdif(\partial B)$.
\end{proof}
\end{lemma}
It follows from Lemma~\ref{lem:boundary of symdif is symdif of boundaries} that if $A\in\cycle$, then
$$
\partial(A\symdif uv)
= (\partial A)\symdif(\partial \{u,v\})
= \emptyset \symdif \{u,v\} =\{u,v\},
$$
and so $A\symdif uv\in\defect$. Conversely, if $A\in\defect$ with $\partial A=\{u,w\}$, then Lemma~\ref{lem:boundary of symdif is symdif of
  boundaries} implies
$$
\partial(A\symdif uv)
= (\partial A)\symdif(\partial \{u,v\}) 
= \{u,w\} \symdif \{u,v\} 
=\begin{cases}
\{w,v\} & w\neq v,\\
\emptyset & w= v,
\end{cases}
$$
which, in turn, implies that either $A\symdif uv\in\defect$ or $A\symdif uv\in\cycle$. Consequently, the worm proposals do indeed yield a
well-defined transition matrix on $\worm$.

To ensure the eigenvalues of the worm transition matrix are strictly positive, we consider a \emph{lazy} version of the metropolized
proposals. This means that at each step, with probability 1/2 we send $A\mapsto A$, and with probability 1/2 we propose a non-trivial
transition, which is then accepted with the appropriate Metropolis acceptance rate. The resulting transition matrix can then be described as follows.
\begin{equation}
{P}_{x}(A, A\symdif uv) := 
  \begin{cases}
    \displaystyle
  x^{\indicator(uv\not\in A)}\frac{1}{2n}\left(\frac{1}{d(u)} +\frac{1}{ d(v)}\right), & A\in\cycle\\
   & \\
    \displaystyle
  x^{\indicator(uv\not\in A)}\frac{1}{4}\left(\frac{1}{d(u)} +\frac{1}{ d(v)}\right), & A\symdif uv \in \cycle\\
   & \\
    \displaystyle
    \min\left(\frac{d(u)}{d(v)}x^{\indicator(uv\not\in A)-\indicator(uv\in A)},1\right)
    \frac{1}{4\,d(u)}, & A\in\defect, \, A\symdif uv \in \defect, u\in \partial A\\
  \end{cases}
  \label{worm transition matrix}
\end{equation}
All other non-diagonal entries of $P_x$ are zero. We refer to $P_x$ with $x=\tanh\beta$ as the worm process \emph{corresponding} to the
Ising model with inverse temperature $\beta$.

By construction, $P_x$ is lazy (and therefore aperiodic) and reversible with respect to $\pi_x$. Corollary~\ref{cor:irreducibility}, to be
discussed in Section~\ref{sec:proof of main theorem}, shows that it is also irreducible. The laziness of $P_x$ ensures that $\rel$ is simply
the reciprocal of the difference between the two largest eigenvalues of $P_x$. For later use, we note the following lower bound.
\begin{lemma} 
\label{lem:pi P bound}
Consider a finite graph $G=(V,E)$ and $x\in(0,1)$. If $A\in\worm$ and $e\in E$, then
$$
  \psi(A)\,\lambda_x(A)\,P_x(A,A\symdif e) \ge \frac{x^{|A\cup e|}}{2\maxdeg}.
$$
\begin{proof}
  If $A\in\cycle$ or $A\symdif uv\in\cycle$, then the result can be seen by inspection from~\eqref{worm transition matrix}. Suppose instead that
  both $A$ and $A\symdif uv$ belong to $\defect$, and that $u\in \partial A$. Then
\begin{align*}
\psi(A)\,\lambda_x(A)\,P_{x}(A,A\symdif uv) 
&= \displaystyle 2\,x^{|A|}\,\min\left(\frac{d(u)}{d(v)}x^{\indicator(uv\not\in A)-\indicator(uv\in A)},1\right)\frac{1}{4\,d(u)}\\
&=
\begin{cases}
\displaystyle
\frac{1}{2\,d(v)}\,x^{|A|+\indicator(uv\not\in A)-\indicator(uv\in A)}, & \text{ if }
\displaystyle \frac{d(u)}{d(v)}x^{\indicator(uv\not\in A)-\indicator(uv\in A)}\le 1,\\
\displaystyle\frac{1}{2\,d(u)}\, x^{|A|}, & \text{otherwise},
\end{cases}\\
&\ge\frac{x^{|A\cup uv|}}{2\maxdeg}.
\end{align*}
\end{proof}
\end{lemma}

\subsection{Bounds on the PS measure}
\label{sec:bounds on PS measure}
We conclude this section with two lemmas concerning the PS measure. Lemma~\ref{lem:minimum pi} is required in our proof of
Theorem~\ref{thm:mixing time bound} in Section~\ref{sec:proof of main theorem}, while Lemma~\ref{lem:pi C0 and pi Cuv lower bounds} is
required in our discussion of fprases in Section~\ref{sec:fpras}.
\begin{lemma}\label{lem:minimum pi}
Consider a finite connected graph $G=(V,E)$ and $x\in(0,1)$. Then $\pi_x(\noedges)\ge2^{-m}$, and for all $A\in\worm$ we have 
$$
\pi_x(A) \ge \frac{1}{2} \left(\frac{x}{2}\right)^m.
$$
\begin{proof}
  Let $Z_{x}:=n\lambda_{x}(\cycle)+2\lambda_{x}(\defect)$ denote the partition function of the PS measure. 
  Then from Corollary~\ref{cor:lambda of Ck over lambda of C0} it follows that
\begin{align*}
  Z_{x} &= n\,\lambda_{x}(\cycle)\left[1+\frac{2}{n}\frac{\lambda_{x}(\defect)}{\lambda_{x}(\cycle)}\right]\\
  &\le n^2\lambda_{x}(\cycle) \\
  &\le n^2\lambda_{1}(\cycle)\\
  &= n^2 \,2^{m-n+1}.
\end{align*}
  To evaluate $\lambda_1(\cycle)$ we have used the fact that for any finite connected graph, the cardinality of the cycle space is $2^{m-n+1}$ (see,
  e.g.~\cite{Diestel05}).

  By definition, $Z_{x}\,\pi_{x}(A)$ equals either $n\,x^{|A|}$ if $A\in\cycle$, or $2\,x^{|A|}$ if $A\in\defect$. Therefore 
  $$
  \pi_x(\noedges) = \frac{n}{Z_x} \ge 2^{-m} \frac{2^{n-1}}{n} \ge 2^{-m}.
  $$
  Likewise, since $Z_{x}\,\pi_{x}(A) \ge 2\,x^{|A|} \ge 2\,x^m$ for all $A\in\worm$, we have
  \begin{align*}
    \pi_x(A) &\ge \frac{2x^m}{n^2 2^{m-n+1}} \\
    &= \frac{1}{2}\left(\frac{x}{2}\right)^m \frac{2^{n+1}}{n^2}\\
    &\ge \frac{1}{2}\left(\frac{x}{2}\right)^m.
  \end{align*}
\end{proof}
\end{lemma}

\begin{lemma} 
  \label{lem:pi C0 and pi Cuv lower bounds}
The PS measure on finite graph $G=(V,E)$ with parameter $x\in(0,1)$ satisfies
  \begin{align*}
    \pi_x(\cycle) &\ge \frac{1}{n},\\
    \pi_x(\sC_{uv}) &\ge \frac{2}{n^2}x^{d(u,v)},
  \end{align*}
  for all $u,v\in V$.
\end{lemma}
\begin{proof}
  We begin with the bound for $\pi_{x}(\cycle)$. Using Corollary~\ref{cor:lambda of Ck over lambda of C0} we obtain
  $$
  \frac{1}{\pi_{x}(\cycle)} = \frac{n\lambda_{x}(\cycle) + 2\lambda_{x}(\defect)}{n\lambda_{x}(\cycle)} 
  = 1 +\frac{2}{n}\frac{\lambda_{x}(\defect)}{\lambda_{x}(\cycle)} \le 1+\frac{2}{n}\binom{n}{2} = n.
  $$
  Likewise, Corollary~\ref{cor:lambda of Ck over lambda of C0} also implies
  \begin{equation}
      \frac{1}{\pi_{x}(\sC_{uv})} 
      = \frac{n}{2}\frac{\lambda_{x}(\cycle)}{\lambda_{x}(\sC_{uv})} + \frac{\lambda_{x}(\defect)}{\lambda_{x}(\sC_{uv})}\\
      \le \frac{n}{2}\frac{\lambda_{x}(\cycle)}{\lambda_{x}(\sC_{uv})} + \binom{n}{2}\frac{\lambda_{x}(\cycle)}{\lambda_{x}(\sC_{uv})}\\
      = \frac{n^2}{2} \frac{\lambda_{x}(\cycle)}{\lambda_{x}(\sC_{uv})}.
      \label{Cuv indicator bound}
      \end{equation}

    Now specify a shortest path $p_{uv}$ between $u$ and $v$, and observe that 
    $$
    |A\symdif p_{uv}|\ge|A|-|p_{uv}|=|A|-d(u,v).
    $$
    Lemma~\ref{cycle space bijection} implies that the map
    $\alpha:\sC_{uv}\to\cycle$ defined by $\alpha(A)=A\symdif p_{uv}$ is a bijection. 
    It follows that
    \begin{equation}
      \lambda(\cycle) = \sum_{A\in\cycle} x^{|A|} 
      = \sum_{A\in\sC_{uv}}x^{|A\symdif p_{uv}|} \le x^{-d(u,v)} \sum_{A\in\sC_{uv}}x^{|A|}
      = x^{-d(u,v)}\,\lambda(\sC_{uv}).
      \label{lambda of C0 over lambda Cuv bound}
    \end{equation}
    The stated result follows by combining~\eqref{Cuv indicator bound} and~\eqref{lambda of C0 over lambda Cuv bound}.
\end{proof}

\begin{lemma}
  \label{cycle space bijection}
  Consider a finite graph $G=(V,E)$, and a set $W\subseteq V$.
  If $F\in\sC_W$, then the map $\alpha:\sC_W\to\cycle$ defined by $\alpha(A) = A\symdif F$ is a bijection.
  \begin{proof}
    Let $A\in\cycle$ and set $A'=A\symdif F$. Lemma~\ref{lem:boundary of symdif is symdif of boundaries} implies that 
    $$
    \partial A'=(\partial A)\symdif(\partial F)=\partial F=W,
    $$
    and so $A'\in\sC_W$. Since $A=\alpha(A')$, this implies that $\alpha$ is surjective.
    
    Now suppose that $\alpha(A)=\alpha(A')$ for $A,A'\in\sC_W$. Then $A\symdif F=A'\symdif F$. But taking symmetric difference of
    both sides with $F$ immediately implies $A=A'$. Therefore $\alpha$ is injective.
\end{proof}
\end{lemma}

\section{Proof of Rapid Mixing}
\label{sec:proof of main theorem}
Consider an irreducible and reversible Markov chain, with finite state space $\Omega$, transition matrix $P$, and stationary distribution
$\pi$. Let $\sG_P=(\Omega,\sE_P)$ denote the \emph{transition graph} of $P$, where $\sE_P=\{(A,A')\in\Omega^2:P(A,A')>0\}$. It is natural to
consider $\sG_P$ to be a directed graph, but we note that since $P$ is assumed reversible, the edges of $\sG_P$ occur in anti-parallel pairs. To
avoid confusion between $\sG_P$ and the underlying graph $G$ on which we define the Ising model, we shall always refer to the elements of
$\sE_P$ as \emph{transitions}, and reserve the word \emph{edge} for elements of the edge set of $G$. Our proof of Theorem~\ref{thm:mixing
  time bound} makes essential use of the following result~\cite[Corollary 3]{Schweinsberg02}. A similar result was proved
in~\cite{JerrumSinclairVigoda04}.
\begin{theorem}[Schweinsberg (2002)~\cite{Schweinsberg02}]
Consider an irreducible and lazy Markov chain, with finite state space $\Omega$ and transition matrix $P$, which is reversible with respect to
the distribution $\pi$. Let $\sS\subseteq\Omega$ be nonempty, and for each pair $(I,F)\in \Omega\times\sS$, specify a path $\gamma_{I,F}$ in
$\sG_P$ from $I$ to $F$. Let
$$
\Gamma=\{\gamma_{I,F}:(I,F)\in \Omega\times\sS\}
$$
denote the collection of all such paths, and let $\sL(\Gamma)$ be the length of a longest path in $\Gamma$. For any transition $T$, let
$$
\cp_T=\{(I,F)\in \Omega\times\sS:\ \gamma_{I,F}\ni T\}.
$$
Then
$$
\rel
\le 
4\sL(\Gamma)\varphi(\Gamma)
$$
where
$$
\varphi(\Gamma) := \max_{(A,A')\in \sE_P} \left\{\sum_{(I,F) \in \cp_{(A,A')}} \frac{\pi(I)\pi(F)}{\pi(\sS)\,\pi(A)\,P(A,A')} \right\}.
$$
\label{Schweinsberg Theorem}
\end{theorem}
\begin{remark}
We note that~\cite{Schweinsberg02} defines the relaxation time to be the reciprocal of the spectral gap, rather than the reciprocal of the
\emph{absolute} spectral gap, as considered here. For this reason, our statement of Theorem~\ref{Schweinsberg Theorem} includes the added
condition that the chains be lazy.
\end{remark}
In the context of the worm process, it is convenient to choose $\sS=\cycle$. This allows the construction of a natural choice of $\Gamma$,
leading to the following result.
\begin{lemma}
\label{congestion of canonical paths lemma}
There exists a choice of paths ${\Gamma=\{\gamma_{I,F}:(I,F)\in\worm\times\cycle\}}$ such that
$$
\varphi(\Gamma) \le \maxdeg\,n^4 \qquad \text{ and } \qquad \sL(\Gamma)\le~m.
$$
\end{lemma}
\begin{corollary}
\label{cor:irreducibility}
The worm process on graph $G$ with parameter $x\in(0,1)$ is irreducible.
\begin{proof}
  Since the worm process is reversible, it suffices to show that it is possible to transition from an arbitrary state $I\in\worm$ to the
  fixed state $\noedges$, in a finite number of steps, with positive probability. Let $\gamma_{I,\noedges}$ denote the path from $I$ to
  $\noedges$ described in Lemma~\ref{congestion of canonical paths lemma}. Since $\sL(\Gamma)\le m$, $\gamma_{I,\noedges}$ corresponds to a
  finite sequence of transitions. By construction, this sequence occurs with positive probability.
\end{proof}
\end{corollary}
\noindent Before we prove Lemma~\ref{congestion of canonical paths lemma}, we use it to prove our main result.
\begin{proof}[Proof of Theorem~\ref{thm:mixing time bound}]
Combining Lemma~\ref{congestion of canonical paths lemma} with Theorem~\ref{Schweinsberg Theorem} immediately implies
\begin{equation}
\rel\le 4\,\maxdeg\,m\,n^4,
\label{relaxation bound}
\end{equation}
as required. To bound the mixing time, we appeal to the following general bound
\begin{equation}
\mix(A,\delta) \le \left[\log\left(\frac{1}{\pi_x(A)}\right) + \log\left(\frac{1}{\delta}\right)\right] \, \rel,
\label{mixing time upper-bounded by relaxation time}
\end{equation}
which holds for any irreducible and reversible finite Markov chain (see e.g.~\cite{Sinclair92}).  Inserting~\eqref{relaxation bound}
together with Lemma~\ref{lem:minimum pi} into~\eqref{mixing time upper-bounded by relaxation time}, then yields the stated bounds for
$\mix(\noedges,\delta)$ and $\mix(\delta)$.
\end{proof}
\noindent It now remains only to prove Lemma~\ref{congestion of canonical paths lemma}.
\begin{proof}[Proof of Lemma~\ref{congestion of canonical paths lemma}]
We begin by specifying a candidate set of paths, $\Gamma$, and then go on to bound $\varphi(\Gamma)$.  To this end, fix an $[n]$-valued
vertex labeling of $G$. The labeling induces a lexicographical total order of the edges, which in turn induces a lexicographical total order
on the set of all subgraphs of $G$. For each cycle, we additionally fix an orientation by demanding that from the lowest labeled vertex in
the cycle we move to the lowest labeled of its neighbors.

In order for the worm process to transition from $I\in\worm$ to $F \in \cycle$, it suffices that it updates, precisely once, those edges in
$G$ which lie in the symmetric difference $I\symdif F$. Since $F\in\cycle$, we have $\partial F=\emptyset$, and so Lemma~\ref{lem:boundary
  of symdif is symdif of boundaries} implies $\partial(I\symdif F)=\partial I$. Suppose that $I\in\defect$ with $\partial I=\{u,v\}$, so
that $\partial(I\symdif F)=\{u,v\}$. As the sum of the degrees of each connected component of $(V,I\symdif F)$ must be even, $u$ and $v$
must belong to the \emph{same} component. Of all the shortest paths in $(V,I\symdif F)$ between $u$ and $v$, let $B_0$ denote the edge set
of the path which appears first in our fixed subgraph ordering for $G$. Now observe that $I\symdif F\setminus B_0\in\cycle$. Every element
of $\cycle$ can be decomposed into a (possibly empty) disjoint union of the edge sets of cycles in $G$ (see e.g.~\cite[\S1.9]{Diestel05}).
Decompose $I\symdif F\setminus B_0$ in this way, order the resulting cycles, and denote them by $B_1,B_2,\ldots,B_k$. Using this
prescription, we therefore obtain a unique disjoint partition $I\symdif F=\cup_{i=0}^kB_i$, where $B_0$ is a path and $B_i$ is a cycle for $i\in[k]$. 

\begin{figure*}[t]
\begin{centering}
\includegraphics[scale=0.65]{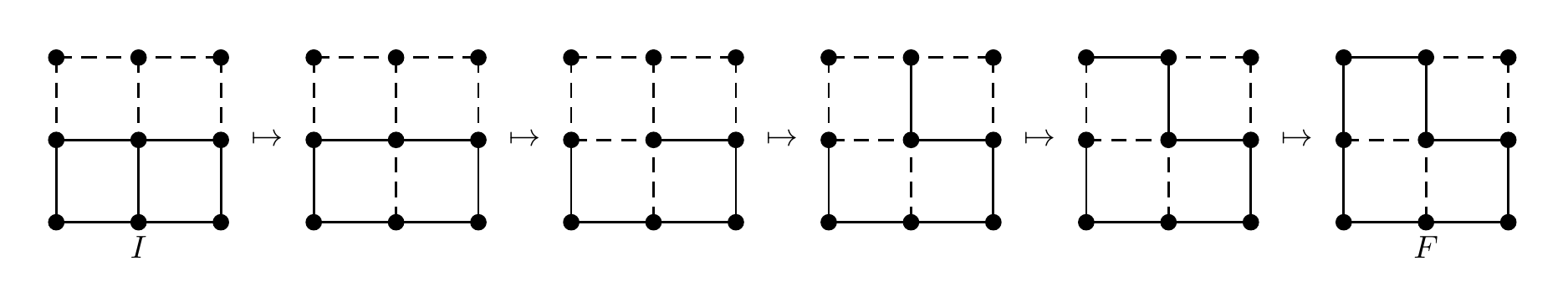}
\end{centering}
\caption{Example of a transition sequence $\gamma_{I,F}$. We order the vertices from left to right, and bottom to top. $I\symdif F=B_0\cup
  B_1$, where the path $B_0$ consists of the single edge $v_2v_5$, and the cycle $B_1$ is $v_4v_5v_8v_7v_4$.}
\label{paths figure}
\end{figure*}

We now define $\gamma_{I,F}$ as follows. The initial state in $\gamma_{I,F}$ is $I$. Beginning from either $u$ or $v$, according to which
has the lowest label, \emph{unwind} the path $B_0$. During this unwinding, the occupation status of each edge in $B_0$ is inverted, in the order in
which it occurs in $B_0$. This produces a sequence of states in $\worm$, the first being $I$, and the last being $I\symdif B_0$, with each
pair of consecutive states differing by a single edge in $B_0$.  Next, unwind the $B_i$ in order, for $i=1,2,\ldots,k$, in each case beginning
with the lowest labeled edge and proceeding according to the fixed cycle orientation. After having unwound $B_k$, we obtain a sequence of
states in $\worm$, the first being $I$ and the last being $I\symdif (\cup_{i=0}^k B_i)=F$, with each pair of consecutive states in the
sequence differing by a single distinct edge in $I\symdif F$. Let $\gamma_{I,F}$ denote this sequence of states. If the initial state $I$ lies in
$\cycle$ rather than $\defect$, the above construction still holds, but with $B_0=\emptyset$, in which case one begins $\gamma_{I,F}$ by
unwinding $B_1$.  Fig.~\ref{paths figure} illustrates a simple example of a transition sequence $\gamma_{I,F}$.

We emphasize that each $\gamma_{I,F}$ constructed according to the above prescription is in fact a path in the transition graph. Indeed,
from~\eqref{worm transition matrix} we see that the transition probability from any state in $\gamma_{I,F}$ to the next state in
$\gamma_{I,F}$ is at least $x/(n\maxdeg)$.  We have therefore constructed a path $\gamma_{I,F}$ from any $I\in\worm$ to any $F\in\cycle$. 
Let $\Gamma=\{\gamma_{I,F}:(I,F)\in\worm\times\cycle\}$ denote the collection of all such paths. Since each edge in $G$ is
updated at most once during the traversal of $\gamma_{I,F}$, we have $\sL(\Gamma)\le m$.

Given this choice of $\Gamma$, we must now bound $\varphi(\Gamma)$. 
Define
$$ 
\cp_{T,k}:=\{(I,F)\in\sC_k\times\cycle:\gamma_{I,F}\ni T\},
$$
and $\cp_{T}:=\cp_{T,0}\cup\cp_{T,2}$. Following arguments similar to those
in~\cite{JerrumSinclair89,Jerrum03}, for each transition $T=(A,A')\in \sE_{P_x}$ we now introduce a map
$\eta_T:\cp_T\to\cycle\cup\defect\cup\doubledefect$ via
$$
\eta_T(I,F) :=\ I\symdif F\symdif (A\cup A').
$$
Since $T$ is of the form $T=(A,A\symdif e)$ for some $e\in E$, we have $\eta_T(I,F) = I\symdif F \symdif(A\cup e)$.  We claim that for each
$T\in \sE_{P_x}$, the map $\eta_T$ is an injection. To show this, we demonstrate how to reconstruct $I$ and $F$ given $\eta_T(I,F)$. The first
observation is that we can recover $I\symdif F$ by simply using $I\symdif F = \eta_T(I,F)\symdif (A\cup e)$.  From our fixed graph labeling,
we can then immediately infer the order in which the edges in $I\symdif F$ must be unwound when traversing the transition path
$\gamma_{I,F}$. Given this information, we can then begin in state $A$ and unwind the remaining edges in $I\symdif F$ specified by
$\gamma_{I,F}$, to recover $F$. We then recover $I$ via ${I=\eta_T(I,F)\symdif(A\cup e)\symdif F}$, and so $\eta_T$ is indeed an
injection.

Next, we prove that for any $T\in \sE_{P_x}$, we can bound the summand appearing in $\varphi(\Gamma)$ by
\begin{equation}
\frac{1}{\pi_x(\cycle)} \frac{\pi_x(I)\pi_x(F)}{\pi_x(A)P_x(A,A')}
\le
\frac{2\maxdeg}{\lambda_x(\cycle)} \psi(I) \lambda_x(\eta_T(I,F)).
  \label{eta bound}
\end{equation}
Let $T=(A,A\symdif e)$ for some $A\in\worm$ and $e\in E$. From Lemma~\ref{lem:pi P bound}, and the fact that $F\in\cycle$, we have
\begin{align*}
\frac{1}{\pi_x(\cycle)} \frac{\pi_x(I)\pi_x(F)}{\pi_x(A)P_x(A,A')}
&=
\frac{1}{n\,\lambda_x(\cycle)} \frac{\psi(I)\,\lambda_x(I)\, n\,\lambda_x(F)}{\psi(A)\lambda_x(A)P_x(A,A')}\\
&\le
\frac{1}{\lambda_x(\cycle)} \psi(I)\,\lambda_x(I)\, \lambda_x(F) \frac{2\maxdeg}{x^{|A\cup e|}}\\
&=
\frac{2\maxdeg}{\lambda_x(\cycle)} \psi(I) x^{|I|+|F|-|A\cup e|} \\
&=
\frac{2\maxdeg}{\lambda_x(\cycle)} \psi(I) x^{|\eta_T(I,F)|}.
\end{align*}
The final equality is a consequence of the following elementary set-theoretic observation, which was utilized in an analogous context
in~\cite{JerrumSinclair93}.  Let $B$ be any subset of $E$ satisfying $I\cap F\subseteq B \subseteq I\cup F$, and set $U=I\symdif F\symdif
B$. It follows that $U\cap B=I\cap F$ and $U\cup B=I\cup F$, and so the inclusion-exclusion principle yields
\begin{equation}
|I|+|F| = |I\cup F| + |I\cap F| = |U\cup B| + |U\cap B| = |U| + |B|.
\label{set theoretic identity}
\end{equation}
Now, by definition, the path $\gamma_{I,F}$ modifies only the edges in $I\symdif F$. Therefore, no edges in $I\cap F$ are updated when
traversing $\gamma_{I,F}$. Conversely, the only edges which can be updated in traversing $\gamma_{I,F}$ are those belonging to $I\cup F$. It
follows that $A\cup e$ satisfies the constraint $I\cap F\subseteq A\cup e\subseteq I\cup F$. Choosing $B=A\cup e$ and $U=\eta_T(I,F)$
in~\eqref{set theoretic identity} then yields 
$$
|I|+|F|-|A\cup e|=|\eta_T(I,F)|,
$$
which establishes~\eqref{eta bound}.

Now let $T=(A,A\symdif e)\in\sE_{P_x}$ be a maximally congested transition. It follows from~\eqref{eta bound} that
\begin{align}
\varphi(\Gamma) 
&\le
 \sum_{(I,F)\in\cp_{T}}
 \frac{2\maxdeg}{\lambda_x(\cycle)} \psi(I) \lambda_x(\eta_T(I,F))
\nonumber\\
 &=
 \frac{2\maxdeg}{\lambda_x(\cycle)}\left[\sum_{(I,F)\in\cp_{T,0}} n\,\lambda_x(\eta_T(I,F)) + \sum_{(I,F)\in\cp_{T,2}}2\,\lambda_x(\eta_T(I,F))\right]
\nonumber\\
 &=\frac{2\maxdeg}{\lambda_x(\cycle)}\left[n\,\lambda_x(\eta_T(\cp_{T,0})) + 2\,\lambda_x(\eta_T(\cp_{T,2}))\right],
\label{decomposed congestion bound}
\end{align}
where the second equality follows from the fact that $\eta_T$ is an injection, and $\eta_T(\cp_{T,k})$ denotes the image of the set
$\cp_{T,k}$ under the map $\eta_T$. By assumption, for any $(I,F)\in\cp_{T,k}$ we have $F\in\cycle$. Since $A\cup e\in\worm$,
Lemma~\ref{lem:boundary of symdif is symdif of boundaries} implies that for any $(I,F)\in\cp_{T,k}$, the state $\eta_T(I,F)$ belongs to
$\worm\cup\doubledefect$. Moreover, if $(I,F)\in\cp_{T,0}$, then in fact $\eta_T(I,F)\in\worm$. Therefore, $\eta_T(\cp_{T,0})\subseteq\worm$
and $\eta_T(\cp_{T,2})\subseteq\worm\cup\sC_4$, and it follows from Corollary~\ref{cor:lambda of Ck over lambda of C0} that
\begin{align*}
  \varphi(\Gamma)
  &\le \frac{2\maxdeg}{\lambda_x(\cycle)}\left[n\,\lambda_x(\sC_0\cup\sC_2) + 2\,\lambda_x(\sC_0\cup\sC_2\cup\doubledefect)\right]\\
  &= 2\maxdeg \left[(n+2) + (n+2)\frac{\lambda_x(\defect)}{\lambda_x(\cycle)} + 2\frac{\lambda_x(\doubledefect)}{\lambda_x(\cycle)}\right]\\
  &\le 2\maxdeg \left[(n+2) + (n+2)\binom{n}{2}+ 2\binom{n}{4}\right]\\
  &\le \maxdeg \, n^4.
\end{align*}
\end{proof}

\section{Fully-polynomial randomized approximation schemes}
\label{sec:fpras}
The rapid mixing of the worm process allows us to construct an efficient randomized approximation scheme for the Ising susceptibility, as well
as the correlation between the spins on any two sites whose distance is not more than some fixed value $k$. These schemes, which simply involve
burning in, and then computing sample means of certain natural random variables, coincide exactly with what a computational physicist
would do in practice.

Before constructing these schemes, we address the issue of how to initialize a worm process. A \emph{cold start} of a Markov chain refers to
starting the process in a fixed initial state. Since the state $\noedges\in\worm$ maximizes $\pi_x(A)$, it is the natural choice of initial
state in a cold start of the worm process. Moreover, since directly constructing arbitrary elements of $\worm$ is a non-trivial task,
initializing a worm process via a more general distribution on $\worm$ is unlikely to be practical in actual simulations. Indeed, perhaps
the simplest way to generate arbitrary elements of $\worm$ is via worm proposals, starting from $\noedges$. Therefore, in this section, we
assume the worm process is started in the fixed state $\noedges$.

We introduce some terminology which will prove convenient below. Consider a positive quantity $x>0$. We say another quantity $\hat{x}$
estimates $x$ with relative error $\epsilon$ if $|\hat{x}-x|\le \epsilon x$. In the case where $\hat{x}$ is random, we say that $\hat{x}$
provides an $(\epsilon,\delta)$-approximation for $x$ if the probability that $\hat{x}$ approximates $x$ with relative error $\epsilon$ is
at least $1-\delta$. In this language, an fpras for an Ising observable is a randomized algorithm which, for arbitrary choices of
$(\epsilon,\delta)$, provides an $(\epsilon,\delta)$-approximation for the given observable, and which runs in a time polynomial in $\epsilon^{-1}$,
$\delta^{-1}$ and $n$.

To establish that our worm estimators yield fprases, we require an appropriate concentration result for the worm process.  There are a
number of Chernoff-type bounds available for finite Markov chains~\cite{Gillman98,Lezaud98}. For our purposes, however, it is convenient to
instead use~\cite[Theorem 12.19]{LevinPeresWilmer09}, which is obtained by bounding the mean square error and then applying Chebyshev's
inequality. Combining this with Theorem~\ref{thm:mixing time bound} we obtain the following.
  \begin{lemma}
    \label{lem:confidence bound}
    Let $(X_t)_{t\in\NN}$ be a worm process on finite connected graph $G$ with parameter $x\in(0,1)$ and $X_0=\noedges$. Let $f:\worm\to\RR^+$, and
    $\epsilon,\delta\in(0,1)$. If
    \begin{align}
    \tau &= \left\lceil4\left(\log(2)+\frac{\log(2\delta^{-1})} {m}\right)\maxdeg\,m^2\,n^4\right\rceil,\label{burn-in bound}\\
       N &\ge \left\lceil16\epsilon^{-2}\delta^{-1}
    \maxdeg\,m\, n^4\, \frac{\|f\|_{\infty}}{\EE_{\pi_x}(f)}\right\rceil,\label{sampling bound}
    \end{align}
    then
    \begin{equation}
    \PP_{\noedges}\left(\left|\frac{1}{N}\sum_{t=\tau+1}^{\tau+N} f(X_t) - \EE_{\pi_x}(f) \right| \ge \epsilon\, \EE_{\pi_x}(f)\right) \le \delta.
    \end{equation}
    In particular, the worm process provides an $(\epsilon,\delta)$-approximation of $\EE_{\pi_x}(f)$ in time of order
    $N+\tau=\displaystyle O\left(\maxdeg\,m\,n^4 \left[m + \epsilon^{-2}\delta^{-1}\frac{\|f\|_{\infty}}{\EE_{\pi_x}f}\right]\right)$.
    \begin{proof}
      The stated result follows almost immediately from Theorem~\ref{thm:mixing time bound} and~\cite[Theorem 12.19]{LevinPeresWilmer09}.
      Although the latter theorem is stated in a form which is uniform over all possible initial states used in a cold start, the proof given
      in~\cite{LevinPeresWilmer09} actually establishes the following slightly sharper result, in which the dependence on the initial state is explicit.
      Specifically, if $\tau\ge\mix(\noedges,\delta/2)$ and $N\ge4\displaystyle\frac{\var_{\pi_x}(f)}{\EE_{\pi_x}(f)^2} \epsilon^{-2}\delta^{-1}\rel$ then 
      $$
      \PP_{\noedges}\left(\left|\frac{1}{N}\sum_{s=0}^{N-1}f(X_{\tau+s})-\EE_{\pi_x}(f)\right|\ge\epsilon\EE_{\pi_x}(f)\right)\le\delta.
      $$
      Since $f$ is assumed positive, we have $\var_{\pi_x}(f)\le \|f\|_{\infty} \EE_{\pi_x}(f)$. The stated result then follows from
      Theorem~\ref{thm:mixing time bound}.
    \end{proof}
  \end{lemma}

We note that there is a simple and standard method, often referred to as the \emph{median trick}, to replace the linear dependence of
$\delta^{-1}$ in the running time of the method in Lemma~\ref{lem:confidence bound} with a logarithmic dependence. Set $\delta=1/4$, and
choose $\tau$ and $N$ as specified in Lemma~\ref{lem:confidence bound}. Use the worm process to generate $k$ independent estimates
$Y_1,\ldots,Y_k$ of $\EE_{\pi_x}(f)$ with these parameters. Choose $k=6\lceil\lg\eta^{-1}\rceil+1$, and let $Y$ be the median of
this set of estimates. Then~\cite{JerrumValiantVazirani86,JerrumSinclair93} we have
$$ 
\PP_{\noedges}(|Y- \EE_{\pi_x} f | \ge \epsilon\, \EE_{\pi_x}f) \le \eta.
$$ 
The choice of $\delta=1/4$ is arbitrary; choosing any other fixed $\delta\in(0,1/2)$ produces a similar result. The constants in the
bound can be sharpened somewhat by choosing $\delta$ more carefully; see~\cite{NiemiroPokarowski09} for a discussion.

\subsection{Susceptibility}
In this section, we apply Lemma~\ref{lem:confidence bound} to the estimation of the susceptibility.
Let $\epsilon\in(0,1)$. We begin with the elementary and general observation that if a quantity $\hat{x}$ estimates a quantity $x$ with
relative error $\epsilon/(1+\epsilon)$, then $1/\hat{x}$ estimates $1/x$ with relative error $\epsilon$. If we construct a quantity
$\hat{S}_0$ which estimates $\pi_x(\cycle)$ with relative error $\epsilon/(1+\epsilon)$, it therefore follows
from~\eqref{susceptibility identity} that $1/\hat{S}_0$ estimates $\chi_{\beta}$ with relative error $\epsilon$.

Let $(X_t)_{t\in\NN}$ be a worm process on finite connected graph $G$ with parameter $x=\tanh(\beta)$ and $X_0=\noedges$.
Since Lemma~\ref{lem:pi C0 and pi Cuv lower bounds} implies
$$
\frac{\|\indicator_{\cycle}\|_{\infty}}{\EE_{\pi_{x}}(\indicator_{\cycle})}=\frac{1}{\pi_{x}(\cycle)} \le n,
$$
setting $f=\indicator_{\cycle}$ and $\epsilon\mapsto \epsilon/(1+\epsilon)$ in Lemma~\ref{lem:confidence bound} shows that if
\begin{equation}
\hat{S}_0 = \frac{1}{N}\sum_{t=\tau+1}^{\tau+N} \indicator_{\cycle}(X_t)
\label{C0 estimator}
\end{equation}
with $\tau$ given by~\eqref{burn-in bound} and 
\begin{equation}
 N= \left\lceil16\epsilon^{-2}(1+\epsilon)^2\delta^{-1} \maxdeg\,m\,n^5 \right\rceil,
\label{N susceptibility}
\end{equation}
then $1/\hat{S}_0$ provides an $(\epsilon,\delta)$-approximation of $\chi_{\beta}$ in time $N+\tau=O(\delta^{-1}\epsilon^{-2}\maxdeg\,m^2\,n^4)$. 
This procedure therefore defines an fpras for $\chi_{\beta}$.

\subsection{Correlation function}
Lemma~\ref{lem:confidence bound} can also be applied to the two-point correlations.  In general, if quantities $\hat{x}$ and $\hat{y}$
estimate quantities $x$ and $y$ with relative error $\epsilon/(2+\epsilon)$, then for any positive constant $c$, the quantity
$c\hat{x}/\hat{y}$ estimates $c x/y$ with relative error $\epsilon$. If $\hat{S}_0$ and $\hat{S}_{uv}$ respectively estimate $\pi_x(\cycle)$
and $\pi_x(\sC_{uv})$, each with relative error $\epsilon/(2+\epsilon)$, it therefore follows from~\eqref{correlation identity} that
\begin{equation}
\frac{n}{2}\frac{\hat{S}_{uv}}{\hat{S}_0}
\label{correlation estimator}
\end{equation}
estimates $\EE_{\pi_x}(\sigma_u\sigma_v)$ with relative error $\epsilon$.

Let $(X_t)_{t\in\NN}$ be a worm process on a finite connected graph $G=(V,E)$ with parameter $x=\tanh(\beta)$ and $X_0=\noedges$. 
We can again estimate $\pi_x(\cycle)$ with $\hat{S}_0$ given by~\eqref{C0 estimator}, and we can likewise estimate
$\pi_x(\sC_{uv})$ with
\begin{equation}
\hat{S}_{uv} := \frac{1}{N}\sum_{t=\tau+1}^{\tau+N}\indicator_{\sC_{uv}}(X_t).
\label{Cuv estimator}
\end{equation}
We emphasize that $\hat{S}_0$ and $\hat{S}_{uv}$ can both be computed from the same realization of $(X_t)_{t\in\NN}$. 
If we demand that $\hat{S}_0$ and $\hat{S}_{uv}$ respectively estimate $\pi_x(\cycle)$ and $\pi_x(\sC_{uv})$ with probability at least
$1-\delta/2$, then the union bound guarantees that~\eqref{correlation estimator} estimates $\EE_{\pi_x}(\sigma_u\sigma_v)$ with probability at least
$1-\delta$.

  Now fix $k\in\NN$. Lemma~\ref{lem:pi C0 and pi Cuv lower bounds} implies that for any $u,v\in V$ with $d(u, v)\le k$, we have
  $$
  \frac{\|\indicator_{\sC_{uv}}\|_{\infty}}{\EE_{\pi_{x}}(\indicator_{\sC_{uv}})}
  =\frac{1}{\pi_{x}(\sC_{uv})} \le \frac{n^2}{2} x^{-k}.
  $$
  Therefore, choosing in~\eqref{Cuv estimator}
  \begin{equation}
  N= \left\lceil16\epsilon^{-2}(2+\epsilon)^2\delta^{-1} \maxdeg\,m\, n^6 x^{-k}\right\rceil
  \label{N correlation}
  \end{equation}
  and $\tau$ as in~\eqref{burn-in bound} with $\delta\mapsto\delta/2$, implies that $\hat{S}_{uv}$ provides an
  $(\epsilon/(2+\epsilon),\delta/2)$-approximation for $\pi_x(\sC_{uv})$.  Since~\eqref{N correlation} is strictly larger than~\eqref{N
    susceptibility}, using this same choice of $N$ and $\tau$ in~\eqref{C0 estimator} also implies that $\hat{S}_0$ provides an
  $(\epsilon/(2+\epsilon),\delta/2)$-approximation for $\pi_x(\cycle)$. Therefore, with this choice of $N$ and $\tau$, the
  estimator~\eqref{correlation estimator} provides an $(\epsilon,\delta)$-approximation for $\EE_{\pi_x}(\sigma_u\sigma_v)$ in time 
  $N+\tau=O(\epsilon^{-2}\delta^{-1}\maxdeg\,m\,n^6 x^{-k})$.

  Finally, let us define the \emph{$k$-restricted} Ising two-point correlation function
  $$
  g_k:\{(u,v)\in V\times V: d(u,v)\le k\}\to\RR
  $$
  by $g_k(u,v)=\EE_{\pi_x}(\sigma_u\sigma_v)$. It follows that, for any fixed $k\in\NN$, the above construction provides an fpras for the
  problem of computing $g_k$.

\begin{acknowledgements}
  The authors wish to gratefully acknowledge the contributions of Greg Markowsky to the early stages of this project, and to also thank Eren
  Metin El\c{c}i, Catherine Greenhill and Alan Sokal for insightful comments on an earlier draft. T.G. also gratefully acknowledges discussions of
  the worm process with many colleagues, particularly Youjin Deng, Catherine Greenhill, Alan Sokal, Boris Svistunov and Ulli Wolff. This
  work was supported under the Australian Research Council's Discovery Projects funding scheme (project numbers DP140100559 \& DP110101141),
  and T.G. is the recipient of an Australian Research Council Future Fellowship (project number FT100100494). A.C. would like to thank STREP
  project MATHEMACS.
\end{acknowledgements}

\bibliographystyle{spmpsci}      

\end{document}